\documentclass[10pt]{amsart}

\usepackage{amsfonts}
\usepackage{latexsym}
\usepackage{pdfsync} 
\usepackage{color}
 \RequirePackage{amsbsy} 
 \RequirePackage{amsopn} 
 \RequirePackage{amsmath}
  \RequirePackage{amssymb}
  \usepackage[mathscr]{eucal}
\definecolor{darkmagenta}{rgb}{0.5, 0, 0.5}
\definecolor{darkblue}{rgb}{0.1, 0.1, 0.7}
\definecolor{darkgreen}{rgb}{0.1, 0.35, 0.1}

 \newtheorem{thee}{Theorem}
 \newtheorem{coor}[thee]{Corollary}
 \newtheorem{leem}[thee]{Lemma}
 
 \newtheorem{prro}[thee]{Proposition}
 
 \newtheorem{exxe}[thee]{Example}
 \newtheorem{reem}[thee]{Remark}

 \newcommand{\balf}
 {\renewcommand{\theenumi}{(\alph{enumi})}
 \renewcommand{\labelenumi}{\theenumi}
                      \begin{enumerate}}
\newcommand{\ealf}   {\end{enumerate}
                      \renewcommand{\theenumi}{\arabic{enumi}}
                      \renewcommand{\labelenumi}{\theenumi.}}
\newcommand{\bara}   {\renewcommand{\theenumi}{(\arabic{enumi})}
                      \renewcommand{\labelenumi}{\theenumi}
                      \begin{enumerate} }
\newcommand{\eara}   {\end{enumerate}
                      \renewcommand{\theenumi}{\arabic{enumi}}
                      \renewcommand{\labelenumi}{\theenumi.}}

 \newcommand{\brom}   {\renewcommand{\theenumi}{(\roman{enumi})}
                      \renewcommand{\labelenumi}{\theenumi}
                      \begin{enumerate} }
\newcommand{\erom}   {\end{enumerate}
                      \renewcommand{\theenumi}{\arabic{enumi}}
                      \renewcommand{\labelenumi}{\theenumi.}}
					
	 \newcommand{\rc}{\color{red}}

	  \newcommand{\mc}{\color{darkmagenta}}

	 \newcommand{\ec}{\color{black}}

	  \newcommand{\Na}{\mbox{\rm Na}}
	  \newcommand{\Kr}{\mbox{\rm Kr}}
	   \newcommand{\Spec}{\mbox{\rm Spec}}
	   \newcommand{\Max}{\mbox{\rm Max}}
	
	   	  \newcommand{\Dbar}{\overline{D}}

     \newcommand{\f}{\boldsymbol{f}}   
       \newcommand{\F}{\boldsymbol{F}}
      \newcommand{\Fbar}{{\overline{\boldsymbol{F}}}} 
	     \newcommand{\starf} {{\star_{_{\! f}}}} 
	      \newcommand{\start} {\widetilde{\star}} 
	           \newcommand{\QMax}{\mbox{\rm QMax}}   
  \newcommand{\co} {\boldsymbol{c}}

\newcommand{\eab}{\texttt{eab} }
\newcommand{\ab}{\texttt{ab} }

  \DeclareMathOperator{\calbV} {\boldsymbol{\mathcal V}}%
  \DeclareMathOperator{\calbW}     {\boldsymbol{\mathcal W}}%
  \RequirePackage{amssymb} 
  \RequirePackage{amsmath} 

\begin{document}

 \title[]{On some classes of integral domains \\ defined by Krull's $\boldsymbol{a.b.}$ operations}

 \author{Marco Fontana and Giampaolo Picozza}


\thanks{\it Acknowledgments. \rm
During the preparation of this paper, the  authors
were  partially supported by  a research grant PRIN-MiUR}

\address{ Dipartimento di Matematica, Universit\`a degli Studi
``Roma Tre'', 00146 Rome, Italy.}
\email{fontana@mat.uniroma3.it }

\email{ picozza@mat.uniroma3.it }

\date{May 2, 2011}


 \maketitle

 \begin{abstract}
 Let $D$ be an integral domain with quotient field $K$. The $b$-operation  that associates to each nonzero  $D$-submodule $E$ of $K$, $E^b :=  \bigcap\{EV \mid V \mbox{ valuation overring of } D\}$, is a semistar operation that plays  an important role in many questions of ring theory (e.g., if $I$ is a nonzero ideal in $D$, $I^b$ coincides with its integral closure).
In a first part of the paper, we study the integral domains that are $b$-Noetherian (i.e., such that, for each nonzero ideal $I$  of $D$, $I^b = J^b$ for some   a finitely generated ideal $J$ of $D$).  For instance, we prove that a $b$-Noetherian domain has Noetherian spectrum and, if it is integrally closed, is a Mori domain, but integrally closed Mori domains with Noetherian spectra are not necessarily $b$-Noetherian.  We also characterize several distinguished classes of $b$-Noetherian domains.
In a second part of the paper, we study more generally the e.a.b. semistar operation of finite type $\star_a$ canonically  associated  to a given semistar operation $\star$ (for instance, the $b$-operation is the
 e.a.b. semistar operation of finite type  canonically  associated to the identity operation). These operations, introduced and studied by Krull, Jaffard, Gilmer and Halter-Koch,  play a very important role in the recent generalizations of the Kronecker function ring. In particular, in the present paper, we classify several classes of integral domains having  some of the fundamental operations $d$, $t$, $w$ and $v$ equal to some of the canonically associated e.a.b. operations $b$, $t_a$, $w_a$ and $v_a$.

  \end{abstract}

\section{Introduction and Background Results}

Let $D$ be an integral domain with quotient field $K$. We denote by $\Fbar(D)$ the set of all nonzero $D$-submodules of $K$, by $\F(D)$ the set of nonzero fractional ideals of $D$ and by $\f(D)$ the set of nonzero finitely generated fractional ideals of $D$. Recall that a {\it star operation on $D$} is a map $\ast : \F(D) \to \F(D)$, $I \mapsto I^\ast$, such that  for
all $\,z \in K\,$, $\,z \not = 0\,$ and for all $\,I,J \in
\F(D)\,$, the following properties hold: $\mathbf{ \bf ({\boldsymbol \ast}_1)}$ \;$D^\ast = D$ and $(zI)^\ast =
zI^\star \,; $ \,
  $\mathbf{ \bf ({\boldsymbol \ast}_2)} \;   I \subseteq J
\;\Rightarrow\; I^\ast
\subseteq
J^\ast \,;$ \,
$\mathbf{ \bf ({\boldsymbol \ast}_3)}\;    I \subseteq
I^\ast
 \textrm {  and  } I^{\ast \ast} := (I^\ast)^\ast = I^\ast$ \cite[Section 32]{G}.

  In \cite{G}, it is shown that if $D$ is an integrally closed domain, the completion of ideals, that is the map which associates to a nonzero fractional ideal $I$ the fractional ideal $I^b  := \bigcap IV$, where $V$ varies over all valuations overrings of $D$, defines a star operation. Note that if $D$ is not integrally closed, the map $b$ still satisfies most of the properties of a star operations; the only problem is that $D^b$ (which coincides with the integral closure of $D$, by Krull Theorem \cite [Theorem 19.8]{G}) is a proper overring of $D$ and is not necessarily a fractional ideal of $D$. This observation leads in a natural way to the notion of semistar operation \cite{OM2}.

  A {\it semistar operation on $D$} is a map $\star : \Fbar(D) \to \Fbar(D)$ such that  for
all $\,z \in K\,$, $\,z \not = 0\,$ and for all $\,E,F \in
\boldsymbol{\overline{F}}(D)$, the following properties hold: $\mathbf{ \bf ({\boldsymbol \star}_1)} \;(zE)^\star =
zE^\star \,; $ \,
  $\mathbf{ \bf ({\boldsymbol \star}_2)} \;   E \subseteq F
\;\Rightarrow\; E^\star
\subseteq
F^\star \,;$ \,
$\mathbf{ \bf ({\boldsymbol \star}_3)}\;    E \subseteq
E^\star
 \textrm {  and  } E^{\star \star} := (E^\star)^\star = E^\star\,. $

 The $b$-operation is a semistar operation even if $D$ is not integrally closed.

 A {\it (semi)star operation on $D$} is a semistar operation $\star$ such that $D^\star = D$, i.e.,   a  (semi)star operation is a semistar operation which restricted to $\F(D)$ is a star operation.
 Conversely, if $\ast$ is a star operation on an integral domain $D$ (hence, defined only on $\boldsymbol{F}(D)$),  we can extend it trivially to a semistar   (in fact, (semi)star) operation on $D$, denoted $\ast_e$, by defining $E^{\ast_e}$ to be the quotient field of $D$ whenever
$E \in \boldsymbol{\overline{F}}(D) \setminus \boldsymbol{F}(D)$.  Therefore, all ``classical'' star operation examples can be considered semistar examples as well.
Note that, in general, a star operation $\ast$ can be extended in several different ways to a semistar operation and $\ast_e$ is just one possible (trivial) way to do so.

As in the classical star-operation setting, we associate to a
semistar  ope\-ra\-tion $\,\star\,$ on $D$ a new semistar
operation
{$\,\star_{_{\! f}}\,$} of $D$ as follows.   If $\,E \in
\boldsymbol{\overline{F}}(D)$, we set:
$$
E^{\star_{_{\! f}}} := \bigcup \{F^\star \;|\;\, F \subseteq E,\, F \in
\boldsymbol{f}(D)
\}\,.
$$
We call {\it$\, {\star_{_{\! f}} }\,$  the
semistar
operation of finite type on $D$  associated to $\,\star\,$}.
If $\,\star =
\star_{_{\! f}}\,$,\ we say that $\,\star\,$ is {\it a
semistar ope\-ra\-tion of finite type on $D$.  \rm  Given two semistar operations $\star'$ and $\star''$ of $D$, we say that $\star' \leq \star''$ if $E^{\star'} \subseteq E^{\star''}$ for all $E \in
\boldsymbol{\overline{F}}(D)$. Note that ${\star_{_{\! f}} } \leq \star$
 and
$\,(\star_{_{\! f}})_{_{\! f}} = \star_{_{\! f}}\,$, \ so $\,\star_{_{\! f}}\,$ is a semistar operation
of finite type of $\,D\,.$

 If $\star$ coincides with the  semistar \emph{$v$--operation of $D$}, defined by $E^{v} : = (D:(D:E))$, for each
   $E \in
\boldsymbol{\overline{F}}(D)$, then  $v_f$ is denoted by $t$. Note that $v$ [respectively, $t$] restricted to $\boldsymbol{{F}}(D)$ coincides with the classical {\sl star} $v$--operation [respectively, $t$--operation]  of $D$. Furthermore, the $v$-semistar operation is the trivial extension of the $v$-star operation. The {\it identity semistar operation on $D$}, i.e., the operation denoted by $d$ and defined by $E^d := E$ for all $E \in \Fbar(D)$, is the smallest semistar operation on $D$. Clearly, $d$, $t$ and $v$ are examples of (semi)star operations on $D$.

 Examples of semistar operations (not necessarily star operations) can be given as follows: Let ${\boldsymbol{\mathcal{T}}} := \{T_\lambda \mid \lambda \in \Lambda \}$
be a family of overrings of $D$, then the operation  $\wedge_{{\boldsymbol{\mathcal{T}}}}$
defined by $E^{\wedge_{{\boldsymbol{\mathcal{T}}}}} :=
 \bigcap\{ ET_{\lambda} \mid \lambda \in \Lambda\}$
 for all $E \in \boldsymbol{\overline{F}}(D)$ is a semistar operation of $D$ and it is a (semi)star operation if and only if $\bigcap\{ T_{\lambda} \mid \lambda \in \Lambda\}= D$. In particular,  if ${\boldsymbol{\mathcal{T}}} := \{T\}$ is a family consisting of a unique proper overring $T$ of $D$, then we have $E^{\wedge_{{\boldsymbol{\mathcal{T}}}}} :=E^{\wedge_{\{T\}}} =
 ET$ is a semistar, but  not a (semi)star, operation of $D$.    If $\iota:D \hookrightarrow T$
 is the canonical embedding of $D$ in one of its overrings and if $\star$ is a semistar operation on $T$,
 we denote by $\star^{\iota}$ the semistar operation defined on $D$ by $E^{\star^\iota} := (ET)^\star$, for all $E\in \Fbar(D)$.
 More generally, it is easy to see that, if $\star_\lambda$ is a semistar operation of $T_\lambda$  and $\iota_\lambda: D \hookrightarrow T_\lambda$
 is the canonical embedding for $\lambda \in \Lambda$, then    $ E^{\wedge{\{\star_\lambda\mid \lambda \in \Lambda\}}}:= E^{\wedge{\{(\star_\lambda)^{\iota_\lambda} \mid \lambda \in \Lambda\}}} :=  \bigcap \{(ET_\lambda)^{\star_\lambda} \mid \lambda \in \Lambda\}$, for all $E\in \Fbar(D)$, defines a semistar operation of $D$.

We say that a nonzero ideal $I$ of $D$ is a
\emph{quasi-$\star$-ideal} if $I^\star \cap D = I$, a
\emph{quasi-$\star$-prime} if it is a prime quasi-$\star$-ideal,
and a \emph{quasi-$\star$-maximal} if it is maximal in the set of
all  proper  quasi-$\star$-ideals. A quasi-$\star$-maximal ideal is  a
prime ideal. It is possible  to prove that each  proper  quasi-$\star_{_{\!
f}}$-ideal is contained in a quasi-$\star_{_{\! f}}$-maximal
ideal.  More details can be found in \cite[page 4781]{FL:2003}. We
will denote by $\QMax^{\star}(D)$ the set of the
quasi-$\star$-maximal ideals  of $D$. By the previous considerations we have that
$\QMax^{\star}(D)$ is not empty, for all semistar operations $\star$ of finite type.  When $\star$ is a (semi)star operation, the condition
$I^\star \cap D = I$ becomes $I^\star =I$ and we simply say that $I$ is a {\it $\star$-ideal} and we will denote by $\Max^{\star}(D)$ the set of the
$\star$-maximal ideals  of $D$ (i.e., the maximal elements in the set of all proper $\star$-ideals).

A semistar operation $\star$ is called {\it stable} if $(E \cap F)^\star = E^\star \cap F^\star$, for all $E, F \in \Fbar(D)$.  By using the localizations at the quasi-$\star_{_{\! f}}$-maximal
ideals, we can associate to $\star$ a semistar operation stable and of finite type, as follows.
For each $E \in \boldsymbol{\overline{F}}(D)$,  we set
  $$
  E^{\start} := \bigcap \left \{ED_Q \mid Q \in  \QMax^{\star_{_{\! f}}}(D) \right\}\,.$$
The previous definition gives rise to  a semistar operation $\start$ on $D$ which is stable  and of finite type, called {\it the semistar operation stable of finite type associated to $\star$}  \cite[Corollary 3.9]{FH}.

Recall that, if $K$ is the quotient field of $D$ and $X$ is an indeterminate over $K$,   the integral domain $\Na(D, \star):= \{f/g \in K(X) \mid  f, g \in D[X], 0 \neq g \mbox{ and } \co(g)^\star =D^\star\}$, overring of $D[X]$,  is called  {\it  the Nagata ring of $D$ associated to the semistar operation $\star$}. It is known that $E^{\start} = E \Na(D, \star) \cap K$ for all $E \in \boldsymbol{\overline{F}}(D)$ \cite[Proposition 3.4(3)]{FL:2003}. It is easy to see that $\start\leq \starf \leq \star$.  We denote by $w$  the semistar operation stable of finite type associated to $v$, i.e., $w := \widetilde{v}$. Since $D^w =D$ { \cite[Proposition 4]{Gr}}, $w$ is also an example of (semi)star operation.

Let
     $\star$ be a semistar operation on $D$. If $F$ is in $\boldsymbol{f}(D)$, we say that
     $F$ is
            \it  $\star$--\texttt{eab} \rm [respectively,  \it
             $\star$--\texttt{ab}\rm]
  if
$(FG)^{\star}
             \subseteq (FH)^{\star}$ implies that $G^{\star}\subseteq H^{\star} $, with $G,\ H \in
\boldsymbol{f}(D)$, [respectively,  with $G,\ H \in
\overline {\boldsymbol{F}}(D)$].

      The operation $\star$ is said to be   \it \texttt{eab} \rm [respectively, \it \texttt{ab}\rm\ \!]  if each $F\in \boldsymbol{f}(D)$ is $\star$--\texttt{eab}  [respectively, $\star$--\texttt{ab}]. \  An \texttt{ab} operation is obviously an \texttt{eab} operation.

  Note  that if $\star$ is an \texttt{eab} semistar operation then $\, {\star_{_{\! f}} }\,$ is also an \texttt{eab} semistar operation, since they agree on all finitely generated ideals.
Note also that, if  $\star$ is a semistar operation of finite type, then
$ \star$  is an \texttt{eab} semistar operation if and only if  $\star$ is an \texttt{ab}  semistar operation; therefore, in the finite type setting, we use the terminology of \texttt{(e)ab}  semistar operation.
In general, we have that
 the notions of   $\star$--\texttt{eab} ideal and   $\star_{_{\! f}}$--\texttt{(e)ab} ideal coincide,    therefore, $\star$ is an \texttt{eab} semistar operation if and only if $\star_f$ is an \texttt{(e)ab} semistar operation  \cite[Proposition 4]{FL:2009}.

Using the fact that, given $F \in \f(D)$,
$F$ is $\star$--\texttt{eab}   if and only if  $\left(
               (FH)^\star : F^\star \right) = H^\star$, for each $H \in
               \boldsymbol{f}(D)$ \cite[Lemma 8]{FL:2009}, we can associate to any semistar ope\-ration $\,\star\,$ on
$\,D\,$  an \texttt{(e)ab} semistar operation of finite type
$\, \star_a\, $  on $\,D\,$, called {\it the \texttt{(e)ab}
semistar
operation associated to $\,\star\,$}, \ defined as follows
for each $ F \in \boldsymbol{f}(D)$ and
for each  $E \in {\overline{\boldsymbol{F}}}(D)$:
$$
\begin {array} {rl}
F^{\star_a} :=& \hskip -5pt  \bigcup\{((FH)^\star:H^\star) \; \ | \; \, \; H \in
\boldsymbol{f}(D)\}\,, \\
E^{\star_a} :=& \hskip -5pt  \bigcup\{F^{\star_a} \; | \; \, F \subseteq E\,,\; F \in
\boldsymbol{f}(D)\}\,,
\end{array}
$$
\cite[Definition 4.4 and Proposition 4.5]{FL1}. \rm  The previous
construction, in the ideal systems setting, is essentially due to {P. Jaffard}  \cite{J:1960}
and {F. Halter-Koch} \cite{HK:1997}, \cite{HK:1998}. The overring $D^{\star_a} = \bigcup\{((H^\star:H^\star) \mid H \in
\boldsymbol{f}(D)\}$ of $D$ is called {\it the $\star$-integral closure of $D$}.
Obviously
 $(\star_{_{\!f}})_{a}= \star_{a}$. Moreover, when $\star =
\star_{_{\!f}}$, then $\star$ is \texttt{(e)ab} if and only if $\star =
\star_{a}\,$ \cite[Proposition 4.5(5)]{FL1}.

For a domain
$\,D\,$ and a semistar operation $\,\star\,$ of $\,D\,$, \ we say
that
a
valuation overring $\,V\,$ of $\,D\,$ is \it a
$\star$--valuation overring of $\,D\,$ \rm provided $\,F^\star
\subseteq FV\,$ (or, equivalently, $\,F^\star V
=FV$)  \ for each $\,F \in \boldsymbol{f}(D)\,$.
Set $\calbV(\star) := \{ V \mid V  \mbox{ is a $\star$--valuation  overring of}$ $D \}$.
 The semistar operation on $D$ defined as follows: for each $E\in \overline{\boldsymbol{F}}(D)$,
$$
E^{{b(\star)}}:= \bigcap\{EV \mid V \in \calbV(\star)\}   = E^{{\wedge_{\calbV(\star)}}} \,,
$$
is an \texttt{ab} semistar operation on $D$ \cite[page 2098]{FL:2009}; clearly, ${b(\star)}= {b(\starf)}$ and ${b(\star)}$ is a (semi)star operation on $D$ if and only if $D$ is $\star$-integrally closed, i.e. $D =D^{\star_a}$.

Let $X$  be an indeterminate over $D$ and $\boldsymbol{c}(h)$ the content of a polynomial $h \in D[X]$.  Then,  we define
 $$
\begin {array} {rl}
\mbox{Kr}(D,\star) := \{ f/g  \mid & \hskip -4pt  f,g \in D[X], \ g
\neq 0,
\;
\mbox{ and there exists }\\ &  \hskip -4pt   h \in D[X] \setminus \{0\}
\; \mbox{
with } \boldsymbol{c}(f)\boldsymbol{c}(h)
\subseteq (\boldsymbol{c}(g)\boldsymbol{c}(h))^\star \,\}.
\end{array}
$$
This is a B\'ezout domain with quotient field $K(X)$, called {\it  the semistar Kronecker function ring  associated to  semistar operation $\star$} \cite[Theorem 5.1 and Theorem 3.11 (3)]{FL1}. Furthermore,  $\mbox{Kr}(D,\star) = \bigcap\{V(X) \mid V \in \calbV(\star)\}$ \cite[Corollary 3.8 and Theorem 5.1]{FL1}.  A key fact is the following (\cite[Corollary 3.4]{FL1} and \cite[Proposition 4.1(5)]{FL:2003}): for each $E\in \Fbar(D)$,
$$
E^{\star_a}=E^{{b(\star)}} = E\mbox{Kr}(D,\star) \cap K \,.
$$

Finally, recall that a nonzero fractional ideal $I$ of $D$ is called {\it $\star$-invertible} if $(II^{-1})^\star = D^\star$
and a domain $D$ is a {\it Pr\"ufer $\star$-multiplication domain} (for short, {\it P$\star$MD})
 if every finitely generated ideal of $D$ is $\starf$-invertible \cite[page 48]{HMM}. For $\star =d$,  a Pr\"ufer $d$-multiplication domain coincides with a Pr\"ufer domain;
 for $\star =v$, the P$v$MD's or  Pr\"ufer $v$-multiplication domains  generalize at the same time Pr\"ufer and Krull domains (\cite{Gr} and \cite{MZ}).

\medskip

After collecting, in Section 2, some properties of the $b$--operation needed later, in Section 3, we study the integral domains that are $b$-Noetherian (i.e., such that, for each nonzero ideal $I$ of $D$, $I^b = J^b$ for some finitely generated ideal $J$ of $D$).  For instance, we prove that a $b$-Noetherian domain has Noetherian spectrum and, if it is integrally closed, is a Mori domain, but integrally closed Mori domains with Noetherian spectra are not necessarily $b$-Noetherian.  We also characterize several distinguished classes of $b$-Noetherian domains and we investigate the local-global behaviour of the $b$--Noetherianity.

In Section 4, we study more generally the \texttt{eab} semistar operation of finite type $\star_a$ canonically  associated to a given semistar operation $\star$ (for instance, the $b$-operation is the
 \texttt{eab} semistar operation of finite type  canonically  associated to the identity operation). These operations, introduced and studied by Krull, Jaffard, Gilmer and Halter-Koch,  play a very important role in the recent generalizations of the Kronecker function ring. In particular, in the present section, we classify several classes of integral domains having  some of the fundamental operations $d$, $t$, $w$ and $v$ equal to some of the canonically associated \texttt{eab} operations $b$, $t_a$, $w_a$ and $v_a$. For instance, the integral domains such that $v$ coincides with $v_a$ [respectively, $t$ coincides with $t_a$; $w$ coincides with $w_a$] (considered as star operations)  are exactly the integrally closed domains such that $\Na(D, v)$ is a divisorial domain [respectively, the $v$--domains; the Pr\"ufer $v$-multiplication domains].
\ec


\section{Elementary properties of the $b$-operation}

{ Given $E \in \Fbar(D)$,} an element $z$ of $K$ is said to be {\it integrally dependent on  $E$} if it satisfies an equation of the form $z^q + a_1 z^{q-1} + \ldots + a_q = 0$, { where $q\geq 1$ and } $a_i \in E^i$ for all $i = 1, 2, \ldots, q$, \cite[Appendix 4, page 349]{ZS}. Equivalently, $z$ is integrally dependent on $E$ if { (and only if) } there exists a { nonzero} finitely generated $D$-submodule $H$ { of $K$} such that $zH \subseteq EH$.

It turns out that the set of the elements that are integrally dependent on $E$ is a $D$-submodule of $K$ which coincides with {\it the completion of $E$}, { i.e., with the $D$-submodule (denoted here by) $E^b$ of $K$ \cite[Appendix 4, Definition 1 and Theorem 1]{ZS}}.

In other words,  for all $E \in \Fbar(D)$:
$$
\bigcup \{(EH : H) \mid H \in  \f(D)\} = E^b =  \bigcap\{EV \mid V \mbox{ valuation overring of } D \}\,.
$$
\ec
 { Therefore, the $b$-operation} coincides with { the \texttt{eab}-semistar operation canonically associated to the identity (semi)star operation, i.e. $b = b(d)$.}
\ec
In particular, the $b$-operation is an \eab (in fact, \ab) semistar operation of finite type. Note that the quasi-$b$-ideals are exactly the  ideals  which are integrally closed (in $D$).

\ec
We collect in the following lemma some elementary facts about the $b$-operation.

\begin{leem} \label{basic} Let $D$ be an integral domain.
\begin{enumerate}
\item[(1)] $b = d_a$.
\item[(2)] $b$ is { an \texttt{ab}} semistar operation of finite type.
\item[(3)]  $P^b \cap D = P$ for each nonzero prime ideal $P$ of $D$.
\item[(4)]  $(\sqrt{I})^b \cap D = \sqrt{I}$ for each nonzero  ideal $I$ of $D$.
\item[(5)]  $\widetilde{b} = d$. In particular,   an ideal  of $D$ is $b$-inver\-tible if and only if   it   is inver\-tible.
\end{enumerate}
\end{leem}
\begin{proof}  For (1) and (2), see the comments preceding Lemma \ref{basic}.

(3) If $P$ is a nonzero prime ideal, there exists a valuation overring $(V, M)$ of $D$ centered in $P$ { \cite[Theorem 19.6]{G}}, and so $P^b \cap D \subseteq PV \cap D \subseteq M \cap D = P$.

 (4)  The previous statement ensures that every nonzero prime ideal is a quasi-$b$-ideal. Radical ideals are also quasi-$b$-ideals, as intersections of quasi-$b$-ideals.

 (5) { The equality} $\widetilde{b} = d$ follows from (3) { (and from the definition of $\widetilde{b}$)}.
 The fact that $b$-invertible ideals are invertible is the consequence of the fact that an ideal is $\starf$-invertible if and only if it is $\start$-invertible (cf. for example \cite[Theorem 2.18]{FP}). { Explicitly, in the present situation, for all $F \in \f(D)$,
 $(FF^{-1})^b = D^b$ is equivalent to $FF^{-1} \not\subseteq M $ for all $M \in \QMax^b(D) = \Max(D)$, i.e., $FF^{-1} =D$.}
 \end{proof}

Let {\bf{SStar}}$(D)$  {[respectively, {\bf{SStar}}${_f}(D)$]} the set of all semistar operations   {[respectively, all semistar operations of finite type]} on $D$.  We can consider the maps $\widetilde{(...)} : $  {\bf{SStar}}$(D) \rightarrow $ {\bf{SStar}}${_f}(D)$, $\star \mapsto \start$, \
and \
$(...)_a:$ {\bf{SStar}}$(D) \rightarrow $ {\bf{SStar}}${_f}(D)$, $\star  \mapsto{\star}_a$.

The relations among\ $\widetilde{(\star_a)}$,\  $(\start)_a$,\
$\start$,\ and\  $\star_a$\ were already investigated in \cite{FL:2003}.  The next goal is to answer  the following  natural question:
when do the maps $(...)_a$ and $\widetilde{(...)}$ establish a bijection { on {\bf{SStar}}${_f}(D)$}?

Note that $\widetilde{(d_a)} = d$, but $(\widetilde{d})_a = b = d_a$, \  and $(\widetilde{b})_a = b$,  but  $\widetilde{(b_a)} = d =\widetilde{b}$. Therefore, using also \cite[Theorem 24.7]{G}, it is easy to verify the next lemma.

\begin{leem} \label{prufer} Let $D$ be an integral domain. The following statements are equivalent.
\begin{enumerate}
\item[(i)] The maps $(...)_a$ and $\widetilde{(...)}$ establish a bijection { on {\bf{SStar}}${_f}(D)$}.

 \item[(ii)] $D$ is a Pr\"ufer domain.

 \item[(iii)]  $d = b$.

 \item[(iv)] {\bf{SStar}}${_f}(D) =\{d\}$.
 \end{enumerate}
 \end{leem}
 \rm \ec

\section{$b$-Noetherian domains}

Recall that an integral domain $D$ is {\it $\star$-Noetherian} if the ascending chain condition
on the quasi-$\star$-ideals of $D$ holds.
If $\star = d$ (where $d$ is the identity (semi)star operation), we have the classical Noetherian domains, if $\star = v$ {  this definition gives back} the Mori domains   \cite[Theorem 2.1]{B:2000},
and if $\star = w$ we obtain the class of strong Mori domains  \cite{WMc1}.
It is well known that a Noetherian domain is characterized by the fact that
each ideal is finitely generated. The semistar version of this characterization
uses the concept of $\starf$-finiteness: if $E \in \Fbar(D)$, we say that $E$ is {\it $\starf$-finite} if
there exists $F \in \f(D)$, $F \subseteq E$, such that $F^{\starf} = E^{\starf}$ (see for instance  \cite[p.650]{FP}).

A $\star$-Noetherian domain is characterized by the fact that each nonzero ideal of $D$ is $\starf$-finite
\cite[Lemma 3.3]{EFP:2004}.

Note that, from this characterization, it follows that D is $\star$-Noetherian if and
only if it is $\starf$-Noetherian.
Finally, we notice that if $\star_1 \leq \star_2$ are two semistar operations on $D$, then $D$ is
$\star_1$-Noetherian implies $D$ is $\star_2$-Noetherian.
Note that the converse does not hold, since a domain can be $\starf$-Noetherian, but not $\start$-Noetherian  \cite[p. 159]{WMc2}.  { When $\star =\start$, i.e. when $\star$ is a  stable semistar operation of finite type, Picozza has shown that several classical properties of Noetherian domains  can be extended to $\start$-Noetherian domains \cite{P:2007}. This is not true for general semistar operations: for instance if $D$ is a (non-integrally closed) $v$-Noetherian domain, then $D[X]$ is not necessarily $v$-Noetherian  \cite[\S 3, Th\'eor\`eme 5]{Querre} and \cite[Theorem 8.4]{Roitman}, but it is true that if $D$ is $w$-Noetherian, then $D[X]$ is  $w$-Noetherian \cite[Theorem 1.13]{WMc2}, \cite[Theorem 4.7]{Park} and \cite[Theorem 2.2]{Chang}.

\begin{reem} \rm When $\ast$ is a star operation on $D$, it is clear that, if $D$ is $\ast$-Noetherian
then $\ast$ is a star operation of finite type (that is, $\ast =\ast_{\!{_f}}$). Indeed, for each nonzero
 { (fractional)} ideal $I$ of $D$ there exists a finitely generated { (fractional)}  ideal $F$ of $D$, $F \subseteq I$, such that $I^\ast = F^\ast$. So,
$I^\ast = F^\ast = F^{\ast_{\!{_f}}}$, thus $I =  I^{\ast_{\!{_f}}}$.

When $\star$ is a semistar operation on $D$, it is still true that $I^{\star}= I^{\star_{\!{_f}}}$ for each nonzero
(fractional) ideal $I$ of $D$, but this is not enough to say that $\star$ is a semistar operation on $D$ of finite type
(even if $\star$ is a (semi)star operation). For instance, let $D$ be a Noetherian domain
that is not conducive  (that is, there exists a proper overring $T$ of $D$, $T \neq K$, such that $T \in  \Fbar(D) \setminus \F(D)$, { i.e., $(D:T) =(0)$  \cite{DF:1984}}). Consider the semistar operation $d_e$ on $D$ defined as follows:
$E^{d_e} := E$ if $E \in \F(D)$ and $E^{d_e} := K$ otherwise (this semistar operation is the so-called
{\it trivial semistar extension of the identity star operation $d$  on $D$}  \cite[Proposition 17]{OM2}). It is
clear that $D$ is $d_e$-Noetherian, but it is easy to check that $d_e$ is not a semistar
operation of finite type.
\end{reem}

\smallskip

In particular, a $b$-Noetherian domain is a domain in which the ascending chain condition on quasi-$b$-ideals hold. Equivalently, since $b$ is a semistar operation of finite type, if for every nonzero { (fractional)} ideal $I$ of $D$, $I$ is $b$-finite, that is, there exists a finitely generated { (fractional)} ideal $F$ (which can be taken inside $I$  by \cite[Lemma 2.3]{FP}, since $b$ is of finite type), such that $F^b = I^b$.

\bigskip

The next goal is to give an example of a $b$-Noetherian domain that is not Noetherian.

{\begin{leem} \label{b-bbar} Let $D$ be an integral domain and $\Dbar$ its integral closure. Set $b :=b_D$ and $\overline{b} := b_{\Dbar}$.
\begin{enumerate}
\item[(1)]  If $\Dbar$ is  $\overline{b}$-Noetherian, then $D$ is $b$-Noetherian.

\item[(2)] If $D$ is $b$-Noetherian and $(D :\Dbar) \neq (0)$, then
  $\Dbar$ is  $\overline{b}$-Noetherian.
  \end {enumerate}
\end{leem}
\begin{proof}  Let $\iota: D \hookrightarrow \Dbar$ be the canonical inclusion. Note that the semistar operation  $\overline{b}^\iota$ on $D$, defined by $E^{\overline{b}^\iota} := (E\Dbar)^{\overline{b}}$ for all $E \in \Fbar(D)$, coincides with $b$, since $\Dbar$ and $D$ have the same valuation overrings. Conversely, for the same reason, the semistar operation $b_\iota$ on $\Dbar$ defined by $E^{{b}_\iota} := E^{{b}}$ for all $E \in \Fbar(\Dbar) \ (\subseteq \Fbar(D))$, coincides with $\overline{b}.$

(1)  It is not difficult to see that if $i: D \hookrightarrow T$ is the canonical inclusion of $D$ in its overring $T$ and if $\star^\prime$ is a semistar operation on $T$, then $T$ $\star^\prime$-Noetherian implies that $D$ is $(\star^\prime)^i$-Noetherian \cite[Lemma 3.1(2)]{EFP:2004}.  For $T = \Dbar$, $\star^\prime = \overline{b}$, and $i = \iota$, if $\Dbar$ is $\overline{b}$-Noetherian, then we conclude that  $D$ is $b$-Noetherian.

(2) Suppose that $D$ is $b$-Noetherian and let $J$ be a nonzero ideal of $\Dbar$. For any $0 \neq x \in (D :\Dbar)$, $I:=xJ$ is a nonzero ideal of $D$ and so for some $F\in \f(D)$, with $F \subseteq I$, $F^b = I^b$. Therefore, if $G :=x^{-1}F\Dbar$, then $G \in \f(\Dbar)$, $G \subseteq J$ and $G^{\overline{b}} = G^b =
x^{-1}(F\Dbar)^b = x^{-1}F^b = x^{-1}I^b = J^b = J^{\overline{b}}$.
\end{proof}

\smallskip

\smallskip

\begin{exxe} \ec { \sl Example of non-Noetherian $b$-Noetherian domain with Noetherian integral closure.}

\rm
Let $D$ be a non Noetherian domain with Noetherian integral closure $\overline{D}$ (e.g., $D := \mathbb{Q} + X \overline{\mathbb{Q}}[\![X]\!]
$, where $\overline{\mathbb{Q}}$ is the field of algebraic numbers, i.e., the algebraic closure of the field of rational numbers  $\mathbb{Q} $ in $\mathbb{C} $.
Note that $D$ is a 1-dimensional Mori non Noetherian non integrally closed local domain with Noetherian spectrum and integral closure $\overline{D} = \overline{\mathbb{Q}}[\![X]\!]$, { which is} a Noetherian domain { \cite[Theorem 3.2]{B:1983}, \cite[Corollary 15(5), Propositions 1.8 and 2.1(7)]{F:1980}}). Since $\Dbar$ is Noetherian, it is $b_{\Dbar}$-Noetherian.
 If $\iota$ is the canonical embedding of $D$ in $\Dbar$,   from Lemma \ref{b-bbar}(1), we have that $D$ is $b_{D}$-Noetherian.
  \end{exxe}
   \ec

\begin{prro} \label{b-noeth} Let $D$ be a $b$-Noetherian domain. Then,
\begin{enumerate}
\item[(1)]  $D$ has Noetherian spectrum.

\item[(2)] Let  $V$ be a valuation overring of $D$. Then, for every { (fractional)}  ideal $I$ of $D$, $IV$ is a principal { (fractional)} ideal of $V$. In particular, if $V$ is an essential valuation overring of $D$, $V$ is a rank 1 discrete valuation domain.

\item[(3)] Assume, moreover, that $D$ is integrally closed, then $D$ is a Mori domain.
\end{enumerate}
\end{prro}
 \begin{proof} (1)
It is  well known that $\Spec(D)$ is Noetherian if and only if each prime ideal of $D$ is the radical of a finitely generated ideal (cf., for instance, \cite[Theorem 3.1.11]{FHP}). Let $P$ be a nonzero prime ideal of $D$.   Since $D$ is $b$-Noetherian, there exists a finitely generated ideal $F$ of $D$,  $F \subseteq P$, such that $F^b = P^b$. Since $F \subseteq \sqrt{F} \subseteq P$, we have $P^b = F^b \subseteq \sqrt{F }^b \subseteq P^b$. By Lemma \ref{basic}(5),  $\sqrt{F} = \sqrt{F}^b \cap D = P^b \cap D = P$ and $P$ is the radical of a finitely generated ideal.

(2)   Let $I$ be a nonzero (fractional) ideal of $D$. By $b$-Noetherianity, there exists a finitely generated { (fractional)} ideal $F$ of $D$, $F \subseteq I$, such that $I^b = F^b$. Thus $IV = I^bV = F^b V = FV = aV$ for some $a \in F$. If, moreover, $V$ is an essential valuation overring of $D$, then $V =D_P$ for some prime ideal $P$ of $D$. So, for each nonzero prime ideal $Q$ of $D$, $Q\subseteq P$, $QD_P$ is a principal ideal of $D_P$.  Therefore, $D_P$ is a Noetherian domain, thus $D_P$ is a rank 1 discrete valuation domain.

 (3) Note that, in this case $D = D^b$, since $D^b$ coincides with the integral closure of $D$, and so $b \leq t$ (see also \cite[Theorem 34.1(4)]{G}), because we already observed that $b$ is a semistar operation of finite type (Lemma \ref{basic}(2)). Therefore, since $D$ is $b$-Noetherian, { $D$} is also $t$-Noetherian  (i.e., Mori).
 \end{proof}

 \begin{reem} \label{remark:conducive}
 \rm
{\bf (a)} By Proposition \ref{b-noeth}(2),  if $V$ is a valuation overring of a $b$-Noetherian domain $D$ such that each ideal
of $V$ is the extension of a fractional ideal of $D$, then $V$ is a DVR. This is the case, for
example, of the following situations:
\begin{enumerate}
\item[(1)]
$(D:V)\neq (0)$ (or, equivalently, $D$ is a {\it  conducive domain} \cite[Theorem 3.2]{DF:1984}, i.e., for every overring $T$ of $D$, $(D:T) \neq (0)$);

\item[(2)] $V$ is flat over $D$ (or, equivalently, $V$ is an essential valuation overring of $D$);

\item[(3)] $V$ well-centered on $D$ (i.e., if each principal ideal of $V$ is generated
by an element of $D$ \cite[page 435]{HR:2004}).
 \end{enumerate}
 The cases   (2) and (3) are considered by
Sega \cite[Proposition 3.8]{S:2007}.

{\bf (b)} From (a), we deduce that a $b$-Noetherian valuation domain (as a valuation overring of a $b$-Noetherian domain) is a DVR. Note that this property can be observed also as a straightforward consequence of the fact that in a valuation domain $b = d$ (Lemma 2).
\end{reem}

Recall that a Noetherian conducive domain (not a field) is one-dimensional and local \cite[Corollary 2.7]{DF:1984}.   More generally, one can easily deduce from \cite[Proposition 3.21]{HR:2004} that the same holds for Mori domains.  From Remark \ref{remark:conducive}, we   obtain  the same result for $b$-Noetherian conducive domains.

\begin{prro} \label{prop:conducive}
Let $D$ be a $b$-Noetherian conducive domain { with quotient field $K$, $D\neq K$}. Then, $D$ is local and one-dimensional.
\end{prro}
\begin{proof}
By Remark \ref{remark:conducive}(a), every valuation overring of $D$ is a DVR. So, $D$ has dimension $1$ { \cite[Theorem 30.8]{G}}. { On the other hand,} conducive domains have at most one prime of height $1$ \cite[Theorem 2.4]{DF:1984}, so $D$ is necessarily local.
\end{proof}

 From Proposition \ref{b-noeth}(2), we deduce immediately the following.


 By the previous Proposition \ref{b-noeth}, for finding an example of a $b$-Noetherian non-Noetherian integrally closed domain, one should look among the examples of (non Noetherian integrally closed) Mori domains with Noetherian spectrum.

Recall that a Pr\"ufer $v$-multiplication domain is characterized by the fact that the localizations at its $t$-maximal ideals are valuation domains  \cite[Theorem 5]{Gr}. Since each domain is intersection of the localizations at its $t$-maximal ideals  \cite[Proposition 4]{Gr}, a P$v$MD is integrally closed.

 It is easy to prove the following.

\begin{prro} \label{PbMD}
\begin{enumerate}
\item[(1)]
Pr\"ufer $b$-multiplication domains coincide with a Pr\"ufer domains.
\item[(2)] The following classes of integral domains coincide:
{
\begin{enumerate}
\item[(i)] $b$-Noetherian Pr\"ufer domains;
\item[(ii)]  $b$-Dedekind domains (i.e., $b$-Noetherian Pr\"ufer $b$-multiplication domains \cite[Proposition 4.1]{EFP:2004});
 \item[(iii)]  Dedekind domains.
\end{enumerate}
}
\end{enumerate}
\end{prro}
\begin{proof}
(1) Recall that  $\widetilde{b} = d$;  {  we have also  observed that a nonzero fractional ideal is $b$-invertible if and only if it is invertible (Lemma \ref{basic}(3)).}

(2) It is enough to recall that an integral domain is
Pr\"ufer if and only if $d =b$ (Lemma \ref{prufer}) and that a Noetherian Pr\"ufer domain is a Dedekind domain.
\end{proof}

In Proposition \ref{PbMD}(2, i), if we replace the assumption  ``Pr\"ufer domain'' with the weaker assumption ``Pr\"ufer $v$-multiplication domain'',  we obtain the following.

\begin{coor} \label{pvmd} A $b$-Noetherian Pr\"ufer $v$-multiplication domain is a Krull domain.
\end{coor}
\begin{proof}
 As observed above, a P$v$MD is integrally closed and an integrally closed $b$-Noetherian domain is Mori (Proposition \ref{b-noeth}(3)). The conclusion follows from the fact that a Mori P$v$MD is Krull \cite[Theorem 3.2]{Kang:1989}.
\end{proof}

In  Theorem \ref{krull} we will further extend the previous corollary.

\begin{reem} \rm
{\bf (a)}   Note that a Krull domain is not necessarily $b$-Noetherian (even if it is always a P$v$MD). For instance, if $K$ is a field and $X_1, X_2, ..., X_n, ...$ is a countable family of indeterminates over $K$, then $D:=K[X_1,$ $ X_2, ..., X_n, ...; n \geq 1]$ is a Krull domain { \cite[Chapitre 7, \S1, Exercice 8]{BAC}}, but it is not $b$-Noetherian, since the ascending chain of prime ($b$-)ideals of $D$ given by
$(X_1) \subsetneq (X_1,X_2) \subsetneq (X_1, X_2, X_3)  \subsetneq \ldots$
is not stationary.  This example (which is, in particular, an integrally closed Mori domain) also shows that the conclusion of statement (3) of Proposition \ref{b-noeth} does not imply $b$-Noetherianity.

{\bf (b)} The conclusion of statement (2) of Proposition \ref{b-noeth} is also  not sufficient to have a $b$-Noetherian domain. Take, for instance, an almost Dedekind domain which is not Dedekind.

{\bf (c)}  A $2$-dimensional valuation domain has Noetherian spectrum but it is not $b$-Noetherian   (Remark \ref{remark:conducive}(b)). Therefore, the conclusion of statement (1) of Proposition \ref{b-noeth}, even in the integrally closed case, does not imply $b$-Noetherianity

{\bf (d)}  From Proposition \ref{b-noeth}(2) { (or, from Corollary \ref{pvmd})} it follows that if $D$ is a $b$-Noetherian P$v$MD then it has $t$-dimension $1$, since the localizations $D_Q$   are DVR's for each $Q \in \Max^t(D)$.
\end{reem}

\begin{exxe} \rm  {\sl Examples  of  Mori integrally closed domains with Noetherian spectrum that are not   $b$-Noetherian}.

 {\bf (a)} Take any DVR $(V, M)$,  let $\pi: V \to V/M$ be the canonical projection. { Assume that $k$ is a proper  subfield of the residue field $\boldsymbol{k}(V) :=V/M$ and that $k$ is algebraically closed in $\boldsymbol{k}(V)$. }The domain $D := \pi^{-1}(k)$ is a non Noetherian integrally closed Mori domain  \cite[Theorem 3.2 or Proposition 3.4]{B:1983} (or, \cite[Theorem 2.2]{B:2000})  and, clearly, Spec$(D)$ is Noetherian, since Spec$(D) = $ Spec$(V)$ { \cite[Corollary 3.11]{AD:1980}}.  However, this domain is not $b$-Noetherian by Proposition \ref{b-noeth}(2), since $D$ admits valuation overrings with non principal extended ideals,  because in the present situation tr.deg$_k(\boldsymbol{k}(V)) \geq 1$. For instance, let $\mathbb{C}$ be the field of complex numbers and let $X$ and $Y$ two inteterminates over $\mathbb{C}$. Take $V:= \mathbb{C}(X)[\![Y]\!]$, $M:= Y\mathbb{C}(X)[\![Y]\!]$, and $k :=\mathbb{C}$. Consider $D := \mathbb{C} + Y\mathbb{C}(X)[\![Y]\!]$.  Clearly, $D$ is an integrally closed { local} 1-dimensional Mori domain with Noetherian spectrum (homeomorphic to $\Spec(\mathbb{C}(X)[\![Y]\!])$).  Set $W := \mathbb{C}[X]_{(X)} + Y\mathbb{C}(X)[\![Y]\!]$. Then, $W$ is a 2-dimensional discrete valuation overring of $D$ with height 1 prime ideal equal to $M = MW$ (in fact, it is not hard to prove that $M$ remains a prime ideal in all the overrings of $D$ included in $V$).  However,  { the prime (nonmaximal) ideal $M $ of $W$ is not a principal ideal, since $W$ is a 2-dimensional discrete valuation domain}.

 {\bf (b)} A nonlocal example of a Mori integrally closed domain with Noetherian spectrum that is not   $b$-Noetherian can be constructed as follows.

Let $K$ be a field and $X,\ Y$ two indeterminates over $K$. Set $D:= K+XK[X,Y]$. It is easy to see that $ D = K[X, Y] \cap
 D_1$, where $D_1:=K + XK[X,Y]_{(X)} = K + XK(Y)[X]_{(X)}$.
By the same arguments used above, $D_1$ is an integrally closed  local 1-dimensional Mori domain with Noetherian spectrum that is not $b$-Noetherian.  By standard properties of the rings of fractions of pullbacks \cite[Proposition 1.9]{F:1980}, it is easy to see that $D_1$ coincides with the localization of $D$ at the maximal ideal $M:= XK[X,Y]$.  Since $D_1$ is not $b$-Noetherian, $D$ also is not $b$-Noetherian (Proposition \ref{wedge}(1)). Furthermore, since for each maximal ideal $N$ of $K[X,Y]$ such that $N \not\supseteq XK[X,Y]$, $D_{N\cap D}$ is canonically isomorphic to $K[X,Y]_N$ and the canonical continuous map $\Spec(K[X,Y]) \rightarrow \Spec(D)$  is surjective \cite[Theorem 1.4 and Corollary 1.5]{F:1980}, then it is easy to conclude that  dim$(D) = 2$,  $D$ has infinitely many maximal ideals of height 2 (different from $M$, which has height 1), and $D$ is a Mori integrally closed domain with Noetherian spectrum  (\cite[Corollary 1.5(5)]{F:1980},  \cite[Proposition 4.5 and Example 4.6(b)]{BG:87}, and also \cite[Example 2]{Lu:2000}).
\end{exxe}

We can look for other possible extensions of Proposition \ref{PbMD}(2). We call a {\it quasi-Pr\"ufer domain}   an integral domain with Pr\"ufer integral closure { \cite[Proposition 1.3]{ACE}}.

\begin{coor} \label{quasi-dedekind} Let $D$ be an integral domain.

\begin{enumerate}
\item[(1)] Assume that the integral closure $  \overline{D}$ of $D$ is a Dedekind domain. Then $D$ is a (one-dimensional) quasi-Pr\"ufer $b$-Noetherian domain.
\item[(2)] Assume that $D$ is a (one-dimensional) quasi-Pr\"ufer $b$-Noetherian domain and that $(D:\overline{D}) \neq (0)$. Then
$  \overline{D}$  is a Dedekind domain.
\end{enumerate}
\end{coor}
\begin{proof} { The statements are easy consequences of Lemma \ref{b-bbar} and Proposition \ref{PbMD}. The condition on the dimension is not essential, since it follows  from the fact that $\dim(D) =1$ if and only if  $\dim(\overline{D})=1$.}
\end{proof}

\medskip

 Recall that, if $D$ is an integral domain, $P$ a prime ideal of $D$, $b$ the $b$-operation of $D$ and $\iota_P$ the canonical embedding of $D$ in $D_P$, then the semistar operation $b_{\iota_P}$ on $D_P$ is defined by { $E^{b_{\iota_P}} :=E^b$, for each $E \in \Fbar(D_P)$.}  Note that $b_{\iota_P}$ coincides with $b_P$, the $b$-operation of $D_P$. Indeed,  clearly $b_{\iota_P} \leq { b_P} $.
 Conversely, since $b$ is an \texttt{ab} semistar operation of finite type on $D$, then $b_{\iota_P} $ is an \texttt{(e)ab} operation of finite type on $D_P$ \cite[Proposition 3.1((1) and (3))]{P:2005} and obviously $d_{P} \leq b_{\iota_P} $, where $d_P$ is the identity operation on $D_P$. Therefore,  $ b_P= (d_{P})_a \leq (b_{\iota_P})_a = b_{\iota_P} $.

\smallskip

\begin{prro} \label{wedge} Let $D$ be an integral domain. For each $P \in \Spec(D)$, denote by $b_P$ the $b$-operation on the localization $D_P$ and by { $ \iota_P: D \hookrightarrow D_P$}  the canonical inclusion.
\begin{enumerate}
\item[(1)]
If $D$ is $b$-Noetherian, then $D_P$ is  { $b_P$-Noetherian} for every $P\in \Spec(D)$.
\item[(2)]  $b = \wedge (b_P)^{\iota_P}$, where $P$ varies over $\Spec (D)$
\end{enumerate}
 \end{prro}
\begin{proof} (1)  Let $J$ be an ideal of $D_P$. Then, $J = ID_P$ for some ideal $I \subseteq D$. Since $D$ is $b$-Noetherian, there exists $F\in\f(D)$, $F \subseteq I$,  such that $F^b = I^b$. Since $b_P = {b_{\iota_P}}$,  then $(FD_P)^{b_P} = (FD_P)^{b} = (F^bD_P)^b = (I^b D_P)^{b} = (ID_P)^b= (ID_P)^{b_P} =J^{b_P}$.
  Therefore, $D_P$ is $b_P$-Noetherian, since $FD_P \in \f(D_P)$.

  (2) Recall that $(b_P)^{\iota_P}$ is {  the semistar operation on $D$ defined by $E^{(b_P)^{\iota_P}} := (ED_P)^{b_P}$ for all $E \in \Fbar(D)$. }  The conclusion follows after observing that $\{V \mid V \mbox{ valuation overring of } D \} = \bigcup \{ \calbW(P) \mid P\in \Spec(D)\}$, where $\calbW(P) := \{ W \mid  W \mbox{ valuation overring of } D_P \}$.  \end{proof}

\medskip

A variation of Proposition \ref{wedge}(2) can be stated for more general   \texttt{eab} semistar operations.

\begin{prro} \label{wedge-max} Let $\star$ be a semistar operation on an integral domain $D$. For each $Q \in \QMax^{\starf}(D)$, let $\iota_Q: D^{\start} \hookrightarrow D_Q$ be the canonical inclusion. Then $$
(\start)_a = \wedge\{(b_Q)^{\iota_Q} \mid Q  \in \QMax^{\starf}(D) \}\,.
$$
{ In particular, for $\star = d $, we have
$$
b = \wedge\{(b_M)^{\iota_M} \mid M  \in \Max(D) \}\,.
$$}
\end{prro}
\begin{proof} Note that $W$ is a valuation overring of $D_Q$ for some $Q \in \QMax^{\starf}(D)$ if and only if $W$ is a $\start$-valuation overring of $D$ \cite[Theorem 3.9]{FL:2003}.

The second part  follows from the first part and from the fact that $d = \widetilde{d} = d_{_f}$, $d_a= b$ and $\QMax^{d}(D) = \Max(D)$.
\end{proof}

%
%

\medskip

The next result generalizes  \cite[Proposition 1.7]{FPT}.

\begin{prro} \label{wedge-lambda-finite}
Let $D$ be an integral domain and $\boldsymbol{\mathcal T}:=
 \{T_\lambda \mid \lambda \in \Lambda\}$ be a family of overrings of $D$.  Assume that  $\boldsymbol{\mathcal T}$ has finite character
  (i.e.,  each nonzero element of $D$  is non-unit in finitely many $T_{\lambda}$'s). For each $\lambda$, let $\star_{\lambda}$ be a given semistar operation of finite type on $T_\lambda$   and let $\iota_{\lambda}: D \hookrightarrow T_\lambda$ be the canonical inclusion. Set ${\boldsymbol \star}
  := \wedge \{(\star_\lambda)^{\iota_\lambda} \mid \lambda \in \Lambda \}$.
  If $I$ is an ideal of $D$ such that $IT_\lambda$ is $\star_\lambda$-finite for each $\lambda$, then $I$ is ${\boldsymbol \star}$-finite.
\end{prro}
\begin{proof} Let $\lambda_1, \lambda_2, ... , \lambda_r$ be the finite set of indexes $\lambda \in \Lambda$  such that $IT_{\lambda} \neq T_{\lambda}$.  Let $G_k \in \f(T_{\lambda_k})$, $G_k \subseteq IT_{\lambda_k}$,  be such that $(G_k)^{\star_{\lambda_k}} = (IT_{\lambda_k})^{\star_{\lambda_k}}$, for $1 \leq k \leq r$.   Now, every generator  $g_i^{(k)}$ of  $G_k \ (\subseteq IT_{\lambda_k})$, for $ 1 \leq i \leq n_k$, can be written as a finite linear combination of elements in $I$ and coefficients in $T_{\lambda_k}$.  Therefore, using all these finite elements of $I$, varying $g_i^{(k)}$ for  all $i$, we can construct $F_k \in \f(D)$, $F_k \subseteq I$, $F_k T_{\lambda_k} = G_k$ for all $k$.
Set $F : = F_1+F_2+ ...+ F_r \in \f(D)$. Then, by a routine argument, it can be shown that $F^{\boldsymbol \star}=I^{\boldsymbol \star}$, with $F \subseteq I$.
\end{proof}

The following corollary is a straightforward consequence of Propositions \ref{wedge}, \ref{wedge-max} and \ref{wedge-lambda-finite}.
\ec

\begin{coor} \label{loc-b}
Let $D$ be an integral domain with the finite character on maximal ideals. The following are equivalent.
\begin{itemize}
\item[(i)] $D$ is $b$-Noetherian.
\item[(ii)] $D_P$ is $b_P$-Noetherian for each prime ideal $P$.
\item[(iii)] $D_M$ is $b_M$-Noetherian for each maximal ideal $M$. \end{itemize}
\end{coor}
\ec

\medskip

In relation with Corollary \ref{quasi-dedekind}, note that if $\Kr(D, b)$ is Noetherian (or, equivalently, Dedekind) then $D$ is $b$-Noetherian. As a matter of fact,
 if $\iota: D \hookrightarrow \Dbar$ is the canonical inclusion, { we have already observed in the proof of Lemma \ref{b-bbar} that $b_\iota$ coincides with $\overline{b} :=b_{\overline{D}}$ (the $b$-operation  on $\overline{D}$).
 Therefore,} $\Kr(D, b) =\Kr(\Dbar, \overline{b})$   \cite[Proposition 4.1(2)]{FL:2003}}; hence,  $\Kr(D, b) $ is Dedekind is equivalent to $\Kr(\Dbar, \overline{b})$ is Dekekind and this happens if and only if $\Dbar$ is Dedekind  \cite[Proposition 38.7]{G}. { Note also that, in this case,} $b =  \wedge_{\{\Dbar\}}$ since, for every $E\in \Fbar(D)$, $E^b = \bigcap \{E\Dbar_N \mid N \in \Max(\Dbar) \} = E\Dbar$.   More generally, the fact that $b = {\wedge_{\{\Dbar\}}}$ characterizes the quasi-Pr\"ufer domains \cite[Remark 2.7]{P:2009}.

 Conversely,  if $D$ is ($b$-)Noetherian, not necessarily $\Kr(D,b)$ is Noetherian. For instance,  take a Noetherian 2-dimensional domain $D$; in this case 2 = dim$(D) = $ dim$_v(D) = $ dim$(\Kr(D,b))$ { \cite[Corollary 30.10 and Proposition 32.16]{G}}. Therefore, the B\'ezout domain  $\Kr(D, b)$ is not Noetherian (since it is not a Dedekind domain).

 \bigskip

 Let $\star_a$ be the \texttt{eab} semistar operation  of finite type canonically associated to a given semistar operation  $\star$ defined on an integral domain $D$. Recall that a domain $D$
  is said {\it $\star$-integrally closed} if $D = D^{\star_a}= \bigcup \{(F^\star : F^\star) \mid  F \in f(D)  \}= \bigcap \{V \mid V $  is a  $\star$-valuation overring of $ D \}$.

In the particular case that $\star =d$,  it is clear that the valuation overrings of $D$  coincide with the $b$-valuation overrings of $D$,
and therefore we reobtain that $D^{d_a} = D^b =  \bigcup \{(F : F) \mid  F \in f(D) \} = \overline{D} = \bigcap \{V \mid V \mbox{ is a valuation overring of } D \} $ { \cite[Theorem 19.8 and Proposition 34.7]{G}}.
 From the previous observations it follows that $D$ is $b$-integrally closed if and only if it is integrally closed and a $\star$-integrally closed domain is always   ($b$-)integrally closed.

\bigskip

\begin{thee} \label{b-1-loc}
{ Let $D$ be an integral domain.
Then, $D$ is a 1-dimensional,  integrally closed, $b$-Noetherian domain if and only if $D$ is a Dedekind domain.}
\end{thee}

The proof of the previous result is based on the following fact of independent interest.

\begin{leem}\label{lemma} Let $D$ be a $b$-Noetherian integrally closed domain and $I$ a nonzero ideal of $D$. Then there exists $m \geq 1$ such that $ (\sqrt{I})^m \subseteq I^b \subseteq \sqrt{I}$. In particular, $\sqrt{I^b} = \sqrt{I}$.
\end{leem}
\begin{proof} We already observed that, in a $b$-Noetherian domain, radical ideals are quasi-$b$-ideals (Lemma \ref{basic}(5)), and so, in $b$-Noetherian integrally closed domain, radical ideals are  $b$-ideals. Therefore, $\sqrt{I} = (a_1, a_2, \ldots, a_t)^b$, for some $a_k \in \sqrt{I}$. Moreover, for each $k$, $1 \leq k \leq t$,  there exists $n_k$ such that $a_k^{n_k} \in I$. Now,  if we take $m := 1 + \sum_{k=1}^t (n_k - 1)$, then  $(\sqrt{I})^{m} \subseteq ((\sqrt{I})^m)^b = (((a_1,a_2, \ldots, a_t)^b)^m)^b = ((a_1,a_2, \ldots, a_t)^m)^b \subseteq I^b$.  Furthermore, if $x \in I^b$ then, for some $n\geq 1$,
$x^n = \sum_{k=1}^{n} a_k x^{n-k}$, with $a_k \in I^k$ and $x^{n-k }\in D^b=\Dbar =D$ and so $x^n \in I\Dbar =I$.

The last statement follows by observing that $\sqrt{(\sqrt{I})^m}= \sqrt{I}$
\end{proof}

\noindent{\it Proof of Theorem \ref{b-1-loc}.}

\noindent  {\bf Claim 1.} {\sl Let $(D, M)$ be an integrally closed, one dimensional,  $b$-Noetherian local domain. Then, $M$ is a principal ideal (i.e., $D$ is a DVR).}

Let  $0 \neq t \in M$. Since $D$ is local one-dimensional, $M = \sqrt{(t)}$. By Lemma \ref{lemma}, there exists an integer $m\geq 1$ such that $(M^m)^b \subseteq (t)^b \subseteq M^b = M$. Since $D$ is integrally closed, { all nonzero principal ideals are integrally closed, in particular  $(t^r)^b = (t^r)$ for any $r\geq 1$.} Thus, $M^m \subseteq (t) \subseteq M$. If $(t) = M$, $M$ is principal and we have done. So, assume that $(t) \subsetneq M$. Since $M^m \subseteq  (t)$, there exists $1 \leq n \leq m$ such that $M^n \subseteq (t)$ but $M^{n-1} \not \subseteq (t)$. Let $a \in M^{n-1} \setminus (t)$, and set $\beta := t/a \in K$.
Note that $\beta^{-1} = a/t \not \in D$, otherwise $a \in tD$. In particular, $\beta^{-1}$ is not integral over $D$. Since $D$ is $b$-Noetherian, there exists a finitely generated ideal  $F$ of $D$ such that $F^b = M$.   If $\beta^{-1}M \subseteq M$, we have that $\beta^{-1} \in (F^b : F^b)$ and so it is in the $b$-integral closure of $D$. { On the other hand, we already recalled above that the $b$-integral closure coincides with the integral closure.}  So, $\beta^{-1}M \subseteq M$ implies that $\beta^{-1} \in \overline D = D$, a contradiction. Thus, $\beta^{-1}M \not \subseteq M$.
{ We claim that} $\beta^{-1}M \subseteq D$. Indeed, $\beta^{-1}M = (a/t)M \subseteq D$, since $aM \subseteq M^{n-1}M = M^n \subseteq (t)$.
Since $\beta^{-1}M \not \subseteq M$ and $D$ is local, we have $\beta^{-1} M = D$ and so $M = \beta D $ is principal.  Therefore,
$D$ is integrally closed, one dimensional and its unique nonzero prime ideal is principal, { hence, by Cohen's Theorem, $D$ is Noetherian \cite[Theorem 3.6]{G}. i.e., a DVR.}

\noindent  {\bf Claim 2.} {\sl Let $D$ be an integrally closed, one dimensional,  $b$-Noetherian domain. Then, $D$ is an almost Dedekind domain.}

 Recall that an {\it almost Dedekind domain} is an integral domain such that $D_M$ is a DVR for all maximal ideals $M$ of $D$.  Therefore, this claim is a consequence of  Proposition  \ref{wedge}(1) and Claim 1.

We now conclude that if $D$ is an integrally closed, one dimensional,  $b$-Noetherian domain, then $D$  is Dedekind by Claim 2 and \cite[Theorem 37.2]{G}, since a $b$-Noetherian domain has the \texttt{acc} on radical ideals (Proposition \ref{b-noeth}(1)), and so every nonzero element is contained in a finite number of minimal primes \cite[Theorem 88]{Ka:1970}, which are also maximal ideals in the present situation.
\hfill $\Box$

\medskip

\begin{reem} \rm Note that, from the properties proved above, \it the following are equivalent:
\begin{itemize}
\item[(i)]  \it $D$ is integrally closed 1-dimensional $b_M$-Noetherian for each $M \in \Max(D)$.
\item[(ii)] \it  $D$ is an almost Dedekind domain.
\end{itemize} \rm
Therefore, a (non-semilocal) integrally closed 1-dimensional domain $D$ which is $b_M$-Noetherian for each $M \in \Max(D)$ is not necessarily $b$-Noetherian, because in this situation $b$-Noetherian coincides with Noetherian (Corollary \ref{loc-b} and Theorem \ref{b-1-loc}).
\end{reem}

\medskip

 By the previous results, the localizations of an integrally closed $b$-Noetherian domain at the primes of height $1$ are DVR's.  So, if one could prove that an integrally closed $b$-Noetherian domain $D$ is the intersection of its localizations at the primes of height $1$, then  $D$ would be a Krull domain.  { We obtain this property by a simple argument, avoiding the techniques used in the proof of Theorem 3.12 of Fossum's book \cite{Fo:1973}.
 Note that next theorem generalizes Corollary \ref{pvmd} and Theorem \ref{b-1-loc}.}

\begin{thee}  \label{krull} Let $D$ be a $b$-Noetherian domain. The following are equivalent.
\begin{enumerate}
\item[(i)] $D$ is integrally closed.
\item [(ii)]  $D$ is completely integrally closed.
\item[(iii)] $D$ is a $v$-domain.
\item[(iv)] $D$ is a P$v$MD.
\item[(v)] $D$ is a  Krull domain.
\end{enumerate}
\end{thee}
\begin{proof} It is well known that (v)$\Rightarrow$(iv)$\Rightarrow$(iii)$\Rightarrow$(i) and (ii)$\Rightarrow$(i) (see, for example, \cite[Section 2]{FZ:2009} and \cite[Theorem 13.1 and page 418]{G}).

In order to show that all the statements are equivalent, it is enough to show that (i)$\Rightarrow$(ii), since a $b$-Noetherian domain integrally closed domain is Mori (Proposition \ref{b-noeth}(3)) and a Mori completely integrally closed domain is Krull {\cite[Exercise 15, page 559]{G}}.

Recall that $D$ is integrally closed (respectively, completely integrally closed) if and only if $D = \bigcup \{ (F:F) \mid F \in \f(D) \}$
(respectively,  $D = \bigcup \{ (I:I) \mid I \in \F(D) \}$) \cite[Theorem 34.3 and Proposition 34.7]{G}.

{ As mentioned above,} if $D$ is integrally closed, it is $b$-integrally closed and conversely. Therefore, $ \bigcup \{ (F:F) \mid F \in \f(D) \} = D = \bigcup \{ (F^b:F^b) \mid F \in \f(D) \}$.
On the other hand, because of the $b$-Noetherianity,  for every $I \in \F(D)$, there exists $F \in \f(D)$ such that $F^b = I^b$. Therefore,  $\bigcup \{ (I:I) \mid I \in \F(D) \} \subseteq
\bigcup \{ (I^b:I^b) \mid I \in \F(D) \}   \subseteq \bigcup \{ (F^b:F^b) \mid F \in \f(D) \} =D$.
\end{proof}

\begin{coor} \label{conducive-b}
A conducive domain is $b$-Noetherian if and only if its integral closure is a DVR. In particular, a conducive integrally closed $b$-Noetherian domain is a DVR.
\end{coor}
\begin{proof}
The ``if'' part follows directly from Lemma \ref{b-bbar}(1).
Conversely, let $D$ be conducive and $b$-Noetherian. Then, $\overline{D}$ is $b$-Noetherian by Lemma \ref{b-bbar}(2) and integrally closed. Moreover, $\overline{D}$ is one-dimensional and local, since $D$ is one-dimensional and local (Proposition \ref{prop:conducive}). Thus $\overline{D}$ is a local Dedekind domain (i.e., a DVR) by Theorem \ref{b-1-loc}.
\end{proof}

\medskip

Recall that a {\it DW-domain} is an integral domain in which $d= w$, that is a domain in which each maximal ideal is a $t$-ideal (see  \cite[Proposition 2.2]{M:2005}} and   \cite[Corollary 2.6]{PT:2008});  a {\it treed domain}  is an integral domain such that $\Spec(D)$  (as a partially ordered set under $\subseteq$) is a tree.
It is well known that} treed domains  (for example,  pseudo-valuation domains or Pr\"ufer domains) are DW-domains   \cite[Corollary 2.7]{DHLZ:1989} and \cite[p. 1957]{PT:2008}. Note that quasi-Pr\"ufer domains, i.e., domains with Pr\"ufer integral closure, are also DW-domains and it is not difficult to give examples of quasi-Pr\"ufer domains that are not treed   \cite[Example 2.28]{Pap:1976}.

  It is also clear that a  Krull DW-domain is a Dedekind domain, and conversely   \cite[Proposition 2.3]{M:2005}.
 As a consequence of Theorem \ref{krull}, the following  result  relates these  classes of domains in the $b$-Noetherian integrally closed case.

\begin{prro} \label{dw-ic}
The following are equivalent.
\begin{itemize}
\item[(i)] $b$-Noetherian integrally closed treed domain.
\item[(ii)] $b$-Noetherian integrally closed DW-domain.
\item[(iii)] Dedekind domain.
\end{itemize}
\end{prro}

\smallskip

Note that the previous result recovers in particular Remark \ref{remark:conducive}(b) and Proposition \ref{PbMD}(2).
The following result is a straightforward consequence of Lemma \ref{b-bbar}(2) and Proposition \ref{dw-ic}.

\begin{coor}
Let $D$ be a $b$-Noetherian domain such that $(D: \overline{D}) \neq (0)$. If $\overline{D}$ is DW, then $\overline{D}$ is a Dedekind domain.
\end{coor}

\smallskip

Note that { the previous result recovers in particular Corollary \ref{quasi-dedekind}(2).  Examples of integral domains whose  integral closure is  a DW  (e.g.,  finite dimensional treed domains, domains with finite spectrum, etc.) are mentioned in \cite[Section 3]{PT:2008}.


\section{Classes of domains defined by $\mbox{\rm {\texttt{eab}}}$ semistar operations}

As the next lemma shows in the star operation case, the equality of the star operations $\ast$ and $\ast_a$ for the classical operations $d, w$ and $t$ characterizes relevant classes of domains. So, it is natural to study more in detail equalities of the previous type, in the general setting of star and semistar operations.

\begin{leem} \label{stara}
\begin{enumerate}
\item[(1)] $d = d_a$ is equivalent to Pr\"ufer domain;
\item[(2)] $w = w_a$ is equivalent to P$v$MD;
\item[(3)]  $t = t_a$ is equivalent to $v$-domain.
\end{enumerate}
\end{leem}
\begin{proof}
(1) is clear from Corollary \ref{prufer}, since $d_a = b$ (Lemma \ref{basic}(1)).

(2) is a consequence of a general characterization of P$\star$MD's given by Fontana-Jara-Santos \cite[Theorem 3.1]{FJS}, from which we have that P$v$MD is equivalent to an integral domain such that $w$ is an \texttt{eab} (semi)star operation.

(3) It is  well known that $v$-domain is equivalent to saying that $v$ is an \texttt{eab} (semi)star operation \cite[p. 418]{G} and thus, also, $t = v_{\!{_f}}$ is an \texttt{(e)ab} (semi)star operation, i.e., $t = t_a$ { \cite[Proposition 4.5(5)]{FL1}}.  Conversely, if $t=t_a$, then it is easy to see that $v$ is \texttt{eab}, since
$(FG)^t =(FG)^v \subseteq (FH)^v = (FH)^t$ implies $G^v = G^t \subseteq H^t = H^v$, for $F,G, H \in \f(D)$.
\end{proof}

\smallskip

\begin{reem}
\rm
In Lemma \ref{stara}, we have (implicitly) considered the equality of two operations {\sl as semistar operations}, that is, we have compared them on $\Fbar(D)$. However, this is not relevant in case of previous lemma,  since the operations considered there are all operations of finite type, so the statements (1), (2) and (3) are respectively equivalent to their analogous ``weaker'' versions (that is, the equality  holds  as {\sl star operations}):

(1$_{\F(D)}$) $d= d_a$ on $\F(D)$,

(2$_{\F(D)}$) $w= w_a$ on $\F(D)$,

(3$_{\F(D)}$) $t= t_a$ on $\F(D)$.

Indeed,   since a finitely generated $D$-submodule of $K$ is always a fractional ideal, the semistar operations of finite type are ``essentially'' defined on $\f(D)$ (since, $E^\star = \bigcup \{F^\star \mid F\in \f(D), F \subseteq E\}$, for each $E \in \Fbar(D)$), that is  if $\star_1$ and $\star_2$ are semistar operations of finite type, then the following are equivalent:

\begin{itemize}
\item[(i)]  $\star_1$ and $\star_2$ coincide on $\f(D)$;
\item[(ii)]  $\star_1$ and $\star_2$ coincide on $\F(D)$;
\item[(iii)]  $\star_1$ and $\star_2$ coincide (on $\Fbar(D))$.
\end{itemize}

As we will see later, when dealing with operations that are not of finite type, the equality as semistar operations is much stronger than the equality as star ope\-rations.
\end{reem}

\smallskip

The next step is the study of domains for which $v = v_a$.    First,  we consider the case when $v = v_a$, as star operations.

\begin{prro} \label{v=va}
Given an integral domain $D$,
 $v = v_a$ on $\F(D)$ if and only if $D$ is  a    P$v$MD with $t$-finite character such that each (nonzero) $t$-prime is contained in only one $t$-maximal ideal and $t$-maximal ideals are $t$-finite (and, therefore, $t$-invertible).
\end{prro}
\begin{proof}
Since $v_a$ is an operation of finite type, then clearly $v = v_a$ (on $\F(D)$) is equivalent to $v = t$ (on $\F(D)$) and $t = t_a$ (on $\F(D)$).
Since $ t = t_a$ on $\F(D)$ is equivalent to $v$-domain (Lemma \ref{stara}(3)), $v = v_a$ on $\F(D)$ is equivalent to $v$-domain which is also a TV-domain (i.e., a domain for which $t=v$  on $\F(D)$,  \cite{HZ:1998} and \cite{E:2009}). The $v$-domains that also are  TV-domains are exactly the P$v$MD's with $t$-finite character such that each (nonzero) $t$-prime is contained in only one $t$-maximal ideal and $t$-maximal ideals are $t$-finite ($t$-invertible) \cite[Theorem 3.1]{HZ:1998}.
\end{proof}

\begin{reem} \rm
 {\bf (a)} The domains with $t$-finite character such that each (nonzero) $t$-prime is contained in only one $t$-maximal ideal are called \emph{weakly Matlis} in \cite{AZ:1999}.  We can say that this is the ``$t$-version'' of Matlis' notion  of {\it \texttt{h}-local domain} (i.e., an integral domain such that each nonzero ideal is contained in at most finitely many maximal ideals and each nonzero prime is contained in a unique maximal ideal \cite{Matlis}).

{\bf (b)}  In particular, a domain in which $v = v_a$ (on $\F(D)$)   is a P$v$MD with the property that $PD_P$ is a principal ideal in the essential valuation overring $D_P$, for every (nonzero) $t$-prime ideal $P$ of $D$. However, a P$v$MD (or, even, a Pr\"ufer domain) with this property not necessarily has $v = v_a$,  even on $\F(D)$. For instance, take an almost Dedekind domain which is non Dedekind.   In this case, $v_a = t_a =w_a= d_a = b = d \lneq v$.
\end{reem}

\medskip

Several  characterizations of domains for which $v = v_a$, as star operations, are summarized in the following proposition. Recall that a {\it domain} is called {\it divisorial} if every nonzero ideal is divisorial (i.e., if $d=v$ as star operations). Heinzer characterized the integrally closed divisorial domains as the \texttt{h}-local
Pr\"ufer domains such that the maximal ideals are finitely generated \cite[Theorem 5.1]{H:1968}.

\begin{prro}\label{v=va-2}
Let $D$ be an integral domain. The following are equivalent.
\begin{itemize}
\item[(i)]    $v = v_a$ on $\F(D)$.
\item[(ii)]   $D$ is a P$v$MD and $v=t$ on $\F(D)$.
\item[(iii)]   $D$ is a $v$-domain  and $v=t$ on $\F(D)$.
\item[(iv)]   $D$ is an essential domain and $v=t$ on $\F(D)$.
\item[(v)]     $D$ is integrally closed and $\Na(D,v)$ is a divisorial domain.
\item[(vi)]   $D$ is integrally closed and $v=w$ on $\F(D)$.
\item[(vii)]   $\Na(D,v) = \Kr(D,v)$ is \texttt{h}-local and the maximal ideals are finitely generated.
\item[(viii)]    $w = v_a$ (on $\F(D)$) and $v=t$ on $\F(D)$.
\item[(ix)]    $v = w_a$ on $\F(D)$.
\item[(x)]  $w=t=v =w_a=t_a=v_a$ on $\F(D)$.
\end{itemize}
\end{prro}
\begin{proof}
(i) $\Rightarrow$ (ii) If $v = v_a$, then $v = t$, since $v_a$ is of finite type. That $D$ is a P$v$MD has been proven in Proposition \ref{v=va}.

(ii)$\Rightarrow$(iv) A P$v$MD is obviously an essential domain.

(iv)$\Rightarrow$(iii) An essential domain is a $v$-domain \cite[Lemma 3.1]{Kang:1989}.

(iii)$\Rightarrow$(ii)  All finitely generated ideals are $v$-invertible, since $D$ is a $v$-domain \cite[Theorem 34.6]{G}. { By assumption}  $v = t$, so all finitely generated ideals are $t$-invertible and $D$ is a P$v$MD.

(ii)$\Rightarrow$(vi) In a P$v$MD, $t = w$ and so if $v = t$, we have $v = t = w$. Moreover, a P$v$MD is integrally closed.

(vi)$\Rightarrow$(ii) $D$ is a P$v$MD by \cite[Theorem 3.3]{EG:2005}. Moreover, $w = v$ implies $t = v$ (since $w \leq t \leq v$).

(vi)$\Leftrightarrow$(v)  \cite[Corollary 3.5 and Proposition 3.2]{GHP:2009}.

(v) $\Rightarrow$(vii) We have already shown that (v) implies (ii), so $D$ is a P$v$MD and  $\Na(D,v) = \Kr(D,v)$ by \cite[Remark 3.1]{FJS}.  Furthermore,  $\Kr(D,v)$ is always a Pr\"ufer domain (in fact, B\'ezout) \cite[Corollary 3.4(2)]{FL1}. Thus, $\Na(D,v)$ is a divisorial Pr\"ufer domain,  hence $\Na(D,v)$ is  \texttt{h}-local with the maximal ideals finitely generated by \cite[Theorem 5.1]{H:1968}.

(vii)$\Rightarrow$(v)  First,   $D$ is a P$v$MD, so integrally closed, by \cite[Remark 3.1]{FJS}. Moreover, $\Na(D,v) = \Kr(D,v)$ implies that $\Na(D,v)$ is Pr\"ufer.  Finally,    a Pr\"ufer \texttt{h}-local domain with the maximal ideals finitely generated is divisorial again by \cite[Theorem 5.1]{H:1968}.

(ii)$\Rightarrow$(x)  It is an easy consequence of the fact that,  in P$v$MD, $t = w  = w_a = t_a \ (=v_a)$.

(x)$\Rightarrow$(viii) This implication is trivial.

(viii) $\Rightarrow$ (ix) We have $v = t \leq t_ a = v_a = w \leq w_a \leq v_a = w \leq t$. Thus, $v = w_a$.

(ix) $\Rightarrow$ (i) $v_a = (w_a)_a = w_a = v$.
\end{proof}

\begin{coor} Let $D$ be an integral domain.
\begin{enumerate}
\item[(1)]  $v = v_a$ on $\F(D)$  and dim$(D)=1$ if and only if $D$ is a Dedekind domain.

\item[(2)]  $v = v_a$ on $\F(D)$  and dim$_t(D)=1$ if and only if $D$ is a Krull domain.
\end{enumerate}
\end{coor}
\begin{proof}   (1) If $D$ is $1$-dimensional, the maximal ideals of $D$ are $t$-ideal, so in $D$ we have $w = d$.  By Proposition \ref{v=va}, we know that, when  $v = v_a$ on $\F(D)$, $D$ is a P$v$MD. In a P$v$MD, we know that  $w = w_a$ (Lemma \ref{stara}(2)), thus
a P$v$MD with $d = w$ is a Pr\"ufer domain, since $b = d_a =w_a = w =d$ {  (Lemma \ref{prufer})}. Moreover,  again from Proposition \ref{v=va},  we have  that $PD_P$ is finitely generated for every (nonzero) ($t-$)prime ideal $P$ of $D$. So $D_P$ is a DVR, for each $P$. Therefore, $D$ is an almost Dedekind domain. On the other hand, we have also that $v = v_a = t = w = b = d$, so $D$ is an almost Dedekind domain in which every nonzero ideal is divisorial, hence a Dedekind domain { since the maximal ideals of $D$ are  finitely generated by \cite[Theorem 5.1]{H:1968}}.
The converse is obvious.

(2)  {\sl Mutatis mutandis}, the proof of this statement follows the lines of the previous proof, using Proposition \ref{v=va} and  recalling that, in a P$v$MD, $D = \bigcap\{D_P \mid P \in \Max^t(D) \}$.  Conversely, in a Krull domain, we have  $ v = t =t_a$ \cite[Corollary 44.3 and Proposition 44.13]{G}.
\end{proof}

\smallskip

Note that a domain in which $\star = \star_a$ is not necessarily a P$\star$MD. For example, in any integral domain $b = b_a$ and,  on the other hand, a P$b$MD is a Pr\"ufer domain.   More generally, for a semistar operation $\star$ of finite type which is \texttt{(e)ab}, we have $\star= \star_a$, however a P$\star$MD is an integral domain for which  $\start= (\start)_a$ \cite[Theorem 3.1]{FJS}.  Therefore,  $\star= \star_a$ does not imply  $\start= (\start)_a$ and, conversely,   $\start= (\start)_a$ does not imply $\star= \star_a$, even on $\F(D)$ (for instance, take $\star = v$ in a P$v$MD which does not verify the other conditions listed in Proposition \ref{v=va}).
\medskip

\begin{reem} \label{v=va semi}\rm Note that, in Proposition \ref{v=va}, we have considered $v = v_a$ as star operations. Suppose now $v = v_a$ as semistar operations, that is,  $E^v = E^{v_a}$ for all $E \in \Fbar(D)$.

In particular, $v = v_a$ as
star operations, so $D$ is a P$v$MD and $w=t=v$. { Assume that $D \neq K$}. Let $V$ be a
$v$-valuation overring of $D$ (since $v = w$ and $D$ is a P$v$MD, one can
take { as $V$ a localization of $D$} at a $t$-maximal ideal). If $(D:V) = (0)$,
$K = V^v = V^w = V$, a contradiction. So, $(D:V) \neq (0)$ and $D$ is
a conducive domain, by \cite[Theorem 3.2]{DF:1984}. So, since $D$ is a conducive integrally closed domain, there exists  a divided prime ideal $P$, such that $D_P$ is a valuation domain \cite[Corollary 4]{BDF:1986}. In particular, $P$ is a $t$-ideal, being the contraction to $D$ of the $t$-ideal $PD_P$ of $D_P$. So, $P$ is a prime $t$-ideal   and, since it is divided, it is comparable to all other prime ideals of $D$ (see, for instance, \cite[proof of Theorem 1]{Ak} and \cite[Proposition 1.2(ii)]{Gb},   or \cite[Proposition 2.1]{D-76}).

 Moreover, since $w=v$, $D$ is weakly Matlis \cite[Theorem 1.5]{EG:2005},
that is, each nonzero element is contained only in a finite number of
$t$-maximal ideals and each $t$-prime is contained in a unique
$t$-maximal ideal. In particular, $P$ is contained in only one
$t$-maximal ideal. But, since $P$ is comparable with all primes of $D$,
it follows that $D$ has exactly one $t$-maximal ideal, say $M$.
{ Since $D$ is a P$v$MD, $D$} is the intersection of the localizations of $D$ at its
maximal $t$-ideals, it follows that $D = D_M$ and so it is a valuation
domain. Furthermore, since  in a valuation
domain $t = d$, we have that $D$ is a divisorial domain and so, in particular, its maximal ideal is principal by \cite[Theorem 5.1]{H:1968}.

{ Conversely, the} fact that if $V$ is a divisorial valuation domain then { it is trivial that $v = v_a$} as
semistar operations.

So we have proven the following result:

\noindent \emph{Let $D$ be an integral domain. Then, $v =
v_a$ as semistar operations if and only if $D$ is a valuation domain
with principal maximal ideal.} \hfill $\Box$
\end{reem}

\medskip

We have already observed  that the integral domains for which $d = b$ are exactly the Pr\"ufer domains (Lemma \ref{prufer}). The next goal is to understand  the domains for which $v = b$.  This is a stronger condition than $v=v_a$, since we require  not only that $v$ is \texttt{eab}   of finite type but, also, precisely that $v_a =b$.
  First, we consider the case when $v = b$ as star operations.

\begin{prro} \label{d=b} Let $D$ be an integral domain. The semistar operations $b$ and $v$ coincide on $\F(D)$ if and only if $D$ is a \texttt{h}-local Pr\"ufer domain such that the maximal ideals are finitely generated or, equivalently, if and only if the (semi)star operations $d$ and $v$ coincide on $\F(D)$  and $D$ is integrally closed.
\end{prro}
\begin{proof} Note that if $b = v$ on $\F(D)$, then in particular $v = v_a$ on $\F(D)$. In this situation,  by Proposition \ref{v=va}, $D$ is a P$v$MD (with further properties). On the other hand, $b = v$ on $\F(D)$ also implies that $D$ is a P$b$MD, i.e., $D$ is Pr\"ufer domain. Furthermore, { since in a Pr\"ufer domain  $d = b$, $D$ is a} divisorial Pr\"ufer domain,  hence we conclude by \cite[Theorem 5.1]{H:1968},  where the second and the third part of the statement are shown to be equivalent.

Conversely it is clear that in a divisorial integrally closed (Pr\"ufer) domain $d =v$ on $\F(D)$ and also, at the same time,  $d =b$ on $\F(D)$, thus $b = v$ on $\F(D)$.
\end{proof}

\begin{reem} \rm  {\bf (a)} Note that, in an integrally closed $b$-Noetherian domain, we have $v = t = w =w_a$ on $\F(D)$, since it is a Krull domain (Theorem \ref{krull}). However, for a general integrally closed $b$-Noetherian (non-Dedekind) domain  $b \lneq w_a$ (Proposition \ref{d=b}).

 {\bf (b)}  If we require $b = v$ as semistar operations (i.e., if we require that $b = v$ on $\Fbar(D)$), we can say something more. In fact, if $V$ is a valuation overring such that $(D:V) = (0)$, we have $V^v = K $ and $V^b = V$. So,  if $b = v$ as semistar operations, such a valuation overring of $D$ cannot exist, { unless $D=K$}. { Therefore, if $D \neq K$,} $D$ must be a conducive domain  \cite[Theorem 3.2]{DF:1984} and, also, by Proposition \ref{d=b}, a Pr\"ufer divisorial domain. Therefore, $D$ is a valuation domain \cite[Lemma 4.6]{P:2005}.  In this case, we also  have $d= v$ as semistar operations.  Moreover,  if $D$ is integrally closed and \ec $d= v$ then $b =v_a$ and, in particular, $b =v$ since  in this case \ec $b \leq t \leq v \leq v_a$.  Therefore, we can conclude, \ec using also Remark \ref{v=va semi},   that {\it for an integral domain $D$, the following are equivalent.
\begin{enumerate}
\item[(i)]  $b= v$ (as semistar operations).
\item[(ii)]  $D$ is a divisorial valuation domain { (i.e., $D$ is a valuation domain with principal maximal ideal).}
\item[(iii)]  $D$ is integrally closed and $d= v$ (as semistar operations).
\item[(iv)] $v = v_a$ (as semistar operations).
\end{enumerate} }
Note that the condition that $D$ is integrally closed is necessary in (iii). For example, in a pseudo-valuation non valuation domain $D$ such that the canonically associated valuation domain $V$ is two-generated  as a $D$-module, $d = v$ on $\F(D)$ (\cite[Corollary 1.8]{HH:1978} and \cite[Proposition 4.3]{HZ:1998}). Moreover, { in this case,} $D$ is conducive, since $(D:V) \neq (0)$, so $d = v$ on $\Fbar(D)$. Clearly, { in this example, $D$ is not integrally closed, since $D^b = V$ and, obviously, $D=D^v \neq D^b$.}
\end{reem}

\bigskip

We consider next the case $w = w_a =b$.

\begin{prro} Let $D$ be an integral domain. The following are equivalent.
\begin{enumerate}
\item[(i)]  $b= w$.
\item[(ii)]  $D$ is a Pr\"ufer domain.
\item[(iii)]  $d = b$.
\end{enumerate}
\end{prro}
\begin{proof} The condition $b = w$ implies that $b$ is a stable semistar operation (of finite type) and so $d =\tilde{b} = b =d_a$.
Henceforth, $D$ is a Pr\"ufer domain. The converse is clear, since in a Pr\"ufer domain not only $d =b$, but also every nonzero finitely generated ideal is invertible (hence, divisorial) and so $d = t$ and thus, in particular, $d=w$
\end{proof}

\medskip
It remains only to study the case when the $b$-operation coincide with the $t$-operation. Recall that a domain is called \emph{vacant} if it  is integrally closed and it has   only one   ``classical'' Kronecker function ring (as in Gilmer's book \cite[page 400]{G}) or, equivalently, if it admits exactly one  \texttt{eab}  star operation of finite type (i.e., the $b$-operation).   For example, any Pr\"ufer domain is a vacant domain.

\begin{prro}  (cf. \cite[Remark 2.9]{F:2010})\label{b=t}
Let $D$ be an integral domain. The following are equivalent:

\begin{enumerate}
\item[(i)]  $b= t$.
 \item[(ii)]    $b = t_a$.
\item[(iii)]  $D$ is a vacant $v$-domain.
\end{enumerate}
\end{prro}
\begin{proof}
(i)$\Rightarrow$(ii) If $b = t$, then in particular  $t$ is \texttt{eab}, and so $t = t_a$  \cite[Proposition 4.5(5)]{FL1}.
 Therefore,  $b = t_a$.

(ii)$\Rightarrow$(iii) Now, let $\ast$ be an \texttt{eab} star operation of
finite type on $D$. Clearly, $\ast = \ast_a \geq d_a = b$, but
also $\ast \leq t$, being $\ast$ a star operation of finite type   \cite[Theorem 34.1(4)]{G}.  Thus, $t_a \geq t \geq \ast = \ast_a \geq b$. Therefore, there is a unique star operation which is  \texttt{eab} and of finite type on $D$. So $D$ is vacant. Moreover,  we have also observed that  $t$ is \texttt{eab} and so $D$ is a $v$-domain   (Lemma \ref{stara}(3)).

 (iii)$\Rightarrow$(i)  Since $D$ is a $v$-domain, $t$ is \texttt{eab}. Since $D$ is vacant, the $b$-operation is the only star operation \texttt{eab} and of finite type. So, $b= t$.
\end{proof}

\begin{reem} \rm
{\bf (a)} Recall that a $t$-integrally closed (= $v$-integrally closed) domain was also called a {\it pseudo-integrally closed} domain in \cite{AHZ}.  With this terminology, the condition (iii) in Proposition \ref{b=t} can be  equivalently stated  (by \cite[Theorem 2.4 and Remark 2.6]{FZ:2009})  in the following form:
\begin{enumerate}
\item[(iii$'$)] {\it $D$ is a vacant pseudo-integrally closed domain.}
\end{enumerate}

Note also that a vacant domain (which is integrally closed by definition) is not in general pseudo-integrally closed (see \cite[Example 12, page 409]{G} and \cite[Proposition 1.8]{AHZ}).

{\bf (b)}  If $D$ verifies   $b = t$,  then $D$ is a (vacant $v$-)domain for which $d \ (= \widetilde{b} = \widetilde{t}\ ) = w$, i.e., a   DW-domain.  More generally,    it can be shown that any vacant domain is a DW-domain \cite[Proposition 2.6]{F:2010}.   Note also that
$$ \mbox{\it DW-domain} \;\; \Leftrightarrow \;\;  b = w_a\,.$$
As a matter of fact, in a DW-domain, $b=d_a =w_a$; conversely,   if  $b = w_a$, then $d = \widetilde{b} = \widetilde{w_a}\geq \widetilde w = w \geq d$.
 Finally, note that the condition  $b = w_a$ is  strictly weaker than the condition
``$D$ is a vacant domain''   (and so, in particular, also of the condition $b =t_a$)   \cite[Example 6.8]{F:2010}.

{\bf (c)}  If, in Proposition \ref{b=t}(iii), we assume P$v$MD [respectively, Krull domain] instead of $v$-domain, we have:
 $$
 \begin{array}{rl}
 D \mbox{ \it is a vacant P$v$MD-domain} &\hskip -5pt\Leftrightarrow D \mbox{ \it is a Pr\"ufer domain}   \Leftrightarrow b =w \ (=w_a) \,,\\
 D \mbox{ \it is a vacant  Krull domain} &\hskip -5pt\Leftrightarrow D \mbox{ \it is a Dedekind domain}  \,.
 \end{array}
 $$
 As a matter of fact, in the first case, $b=t_a$  and $w =w_a$  easily imply that $b =w_a$
 and so $b = w_a = w$.  Conversely, from $b = w$, clearly $w =w_a$, and also $d = \widetilde{b} = \widetilde{w} = w$ and hence $d =b$.
 For the second case, it is sufficient to recall that   a Krull domain is a P$v$MD and that Krull Pr\"ufer domain is a Dedekind domain.
\end{reem}

\end{document}